\documentclass[11pt]{amsart}

\usepackage{amssymb,amsthm,amsmath, bm}
\usepackage{hyperref}
\usepackage[alphabetic]{amsrefs}
\usepackage[shortlabels]{enumitem}
\usepackage[margin=1.5in]{geometry}

\newtheorem{theorem}{Theorem}
\newtheorem{corollary}[theorem]{Corollary}
\newtheorem{lemma}[theorem]{Lemma}

\theoremstyle{remark}
\newtheorem{remark}[theorem]{Remark}

\newcommand{\St}{\mathrm{S}}
\newcommand{\T}{\mathrm{T}}
\newcommand{\N}{\mathbb{N}}

\newcommand{\C}{\mathbb{C}}
\newcommand{\Mix}{\mathrm{Mix}}
\newcommand{\Max}{\mathrm{Max}}

\renewcommand{\epsilon}{\varepsilon}
\renewcommand{\leq}{\leqslant}
\renewcommand{\geq}{\geqslant}

\date{}
\author{Abhinav Chand}
\address[Abhinav Chand]{Department of Mathematics, University of Louisiana at Lafayette, Lafayette, 70506, USA}
\email{abhinav.chand1@louisiana.edu}

\author{Leonel Robert}
\address[Leonel Robert]{Department of Mathematics, University of Louisiana at Lafayette, Lafayette, 70506, USA}
\email{lrobert@louisiana.edu}

\author{Arindam Sutradhar}
\address[Arindam Sutradhar]{Department of Mathematics, University of Louisiana at Lafayette, Lafayette, 70506, USA}
\email{arindam.sutradhar1@louisiana.edu}

\title{Simultaneous averaging to zero by unitary mixing operators}

\begin{document}
\maketitle

\begin{abstract}
We show that if every element a vector subspace of a C*-algebra can be averaged to zero by means 
of unitary mixing operators, then all the elements of the subspace can be simultaneously averaged to zero by a net of 
unitary mixing operators. Moreover, such subspaces admit a simple description in terms of commutators
and  kernels of states on the C*-algebra. We apply this result to  center-valued expectations in C*-algebras with the Dixmier property. 	
\end{abstract}
	
\section{Introduction}

Let $A$ be a unital C*-algebra. Let $U(A)$ denote its unitary group. We call a linear operator $T\colon A\to A$   a unitary mixing (or averaging) operator if it is a convex combination of inner automorphisms of $A$. That is, 
\[
T = \sum^n_{j=1} t_j\mathrm{Ad}_{u_j},
\]
where $u_j\in U(A)$ and $t_j\geq 0$ for $1\leq j\leq n$, $\sum^n_{j=1} t_j=1$, and where  $\mathrm{Ad}_u(x)=uxu^*$ for all $x\in A$. There is a large literature, going back to Dixmier's approximation theorem for von Neumann algebras, around the averaging of elements of  a C*-algebra by means of unitary mixing operators. The question that we investigate here is that of simultaneously averaging a collection of elements of a C*-algebra towards the zero element. 

Let $\St(A)$ denote the states space of $A$. Let $\Max(A)$ denote the maximal (two-sided) ideals space of $A$. We prove the following theorem:
\begin{theorem}\label{zeroaverage}
Let $A$ be a unital C*-algebra. Let $V\subseteq A$ be a linear subspace. 
The following are equivalent

\begin{enumerate}[(i)]
\item	
$V\subseteq \overline{[A,A]}$ and for each maximal ideal $M\in \Max(A)$ there exists $\rho_M\in \St(A)$ that factors through $A/M$ and such that $V\subseteq \ker \rho_M$.

\item
For each $v\in V$ there exists a sequence of unitary mixing operators $(T_n)_{n=1}^\infty$ such that
$T_nv\to 0$. 
\item
There exists a net of unitary mixing operators $(T_\lambda)_\lambda$ such that $T_\lambda v\to 0$ for all  $v\in V$.	
\end{enumerate}
\end{theorem}	

It is not difficult to derive the equivalence of (i) and (ii) from \cite[Theorem 4.7]{NRS}. Our main contribution here is  that (i) implies (iii). Observe that in (i), the selection of states $M\mapsto \rho_M$ may be done only
on the maximal ideals $M$ such that $A/M$ has no bounded traces, for if $A/M$ has a trace, then we are guaranteed the 
existence of a $\rho_M$ vanishing on $V$ by the fact that $V\subseteq \overline{[A,A]}$. 

Recall that a C*-algebra is said to have the Dixmier property if for each $a\in A$ there exist unitary mixing operators $(T_n)_{n=1}^\infty$ such that $T_na\to b\in Z(A)$, where $Z(A)$ denotes the center of $A$. The above mentioned approximation theorem of Dixmier states that every von Neumann algebra has the Dixmier property. This is not the case, however, for every C*-algebra (see \cite{ART}). As an immediate corollary of the previous theorem, we obtain the following:  

\begin{corollary}\label{Hmaps}
Let $A$ be a unital C*-algebra with the Dixmier property. Let $H\colon A\to Z(A)$ be a $Z(A)$-linear, positive, unital map such that 
\begin{enumerate}[(a)]
\item
$\tau\circ H =\tau$ for all bounded traces $\tau$ on $A$,
\item
$H(M)\subseteq M\cap Z(A)$ for all $M\in \Max(A)$.
\end{enumerate} 
Then there exists a net of unitary mixing operators  $(T_\lambda)_\lambda$  such that $T_\lambda\to H$  in the point-norm topology.
\end{corollary}

This result is possibly well known when $A$ is a von Neumann algebra, although we are not aware of a reference for this case. A related theorem of Magajna for weakly central C*-algebras replaces unitary mixing operators by C*-convex combinations (\cite[Corollary 1.2]{magajna}). The special case of Corollary \ref{Hmaps} when $A$ is simple and traceless is obtained by Zsido in \cite{zsido}. We rely on Zsido's result to prove Theorem \ref{zeroaverage} above.

\section{Preliminaries on mixing operators and Dixmier sets}
Let $A$ be a unital C*-algebra. Fix $n\in \N$. Consider the C*-algebra $A^n$, i.e., the direct sum of $n$ copies of $A$. (The norm in $A^n$ is the maximum norm: $\|\bm a\|=\max \|a_j\|$.) We refer to the elements of $A^n$ as $n$-tuples. We regard $A$ embedded in $A^n$ as the constant $n$-tuples.  

Let us denote by $\Mix(A)$ the set of unitary mixing operators $T\colon A\to A$, as defined in the introduction.
We consider unitary mixing operators in $\Mix(A)$ acting on $n$-tuples coordinatewise: given $T\in \Mix(A)$ and 
$\bm{a} =(a_1,\ldots,a_n)\in A^n$, we write
\[
T(\bm a)=(Ta_1,\ldots,Ta_n).
\]
Given an $n$-tuple $\bm{a} =(a_1,\ldots,a_n)$ in $A^n$, let us define
\begin{align*}
D_A(\bm{a}) = \overline{\{T\bm a: T\in \Mix(A)\}}^{\|\cdot\|},
\end{align*}
which we call the Dixmier set generated by $\bm a$ relative to $A$.

The  next two lemmas are essentially obtained in \cite{archbold}, but we establish them here in the form that will be needed later on.

\begin{lemma}[Cf. {\cite[Proposition 2.4]{archbold}}]
Each  operator in $\Mix(A)$ is a limit in the point-norm topology of a net of unitary mixing operators whose unitaries are exponentials, i.e., operators of the form
\begin{equation}\label{expos}
\sum^n_{j=1} t_j\mathrm{Ad}_{e^{ih_j}},
\end{equation}
where $h_j\in A$ is selfadjoint for all $j$.
\end{lemma}
\begin{proof}
It suffices to approximate $\mathrm{Ad}_u$, with $u\in U(A)$,   in the point-norm topology by unitary mixing operators of the form \eqref{expos}.
Let $\{x_1,\ldots,x_n\}\subseteq A$. By the Borel functional calculus in $A^{**}$,  there exists a selfadjoint $h\in A^{**}$ such that $u=e^{ih}$. By Kaplansky's density theorem, there exists a bounded net $(h_\lambda)_\lambda$ of selfadjoint elements in $A$ such that $h_\lambda\to h$ in the ultrastrong* topology. Set $u_\lambda=e^{ih\lambda}$. Then 
$u_\lambda x_i u_\lambda^*\to u x_i u^*$ in the ultrastrong* topology, and thus also in the $\sigma(A^{**},A^*)$ topology for all $i$. Thus, $u_\lambda \bm x u_\lambda^*\to u\bm x u^*$  in the $\sigma((A^{n})^{**},(A^n)^{*})$-topology, where we have set  $\bm x=(x_1,\ldots,x_n)$. It follows,  by a standard application of the Hahn-Banach Theorem, that $u\bm x u^*$ belongs to the norm closure of $\mathrm{co}\{u_\lambda \bm x u_\lambda^*:\lambda\}$. This yields, for each $\epsilon>0$, an  operator $T\in \Mix(A)$ of the form \eqref{expos}   and  such that $\|T(\bm x)-u\bm x u^*\|<\epsilon$, as desired.
\end{proof}		

Recall that $A$ is a said to have the Dixmier property if $D_A(a)\cap Z(A)\neq \varnothing$ for all $a\in A$, where $Z(A)$ denotes the center of $A$. 

\begin{lemma}\label{centermapsdensely}
Let $A$ be a C*-algebra with the Dixmier property. Let $I$ be a closed two-sided ideal of $A$.  Let $\bm a\in A^n$. Then  $D_A(\bm a)\cap Z(A)^n$ is mapped densely onto $D_{A/I}(\pi(\bm a))\cap Z(A/I)^n$ by the quotient map $\pi\colon A\mapsto A/I$. 
\end{lemma}
\begin{proof}
We follow arguments from  \cite{archbold} adapted to $n$-tuples.
	
Let $\bm {\tilde z}\in D_{A/I}(\pi(\bm a))\cap Z(A/I)^n$. Let $\epsilon>0$. Then $\|\tilde T \bm a-\bm{\tilde z}\|<\epsilon$
for some  $\tilde T\in \Mix(A/I)$. Moreover, by the previous lemma, we may choose $\tilde T$ of the form \eqref{expos}.
Since the unitaries in $\tilde T$ are exponentials, they lift to unitaries in $A$. In this way  we get $T\in \Mix(A)$	such that $\pi T = \tilde T$, where $\pi\colon A\to A/I$ denotes the quotient map. Set $\bm b=T(\bm a)$. Then $\|\pi(\bm b)-\bm z\|<\epsilon$. By a process of successive averagings by unitary mixing operators, we can find $T_n\in \Mix(A)$ such that $T_n\bm b\to \bm z\in D_A(\bm b)\cap Z(A)^n$. This is \cite[Lemma 8.3.4]{kad-ring}, stated for the von Neumann algebra case, but the proof applies without change to any C*-algebra with the Dixmier property. Then $\bm z\in D_A(\bm a)$ and $\|\pi(\bm z)-\bm{\tilde z}\|<\epsilon$, as desired.
\end{proof}

\begin{lemma}\label{centrallyconvex}
Let $A$ be a von Neumann algebra. Then  $D_A(\bm a)$ has the following central convexity property: $z\bm b+(1-z)\bm c\in D_A(\bm a)$ for all $\bm b,\bm c\in D_A(\bm a)$ and  $0\leq z\leq 1$ in $Z(A)$.
\end{lemma}

Note: This lemma is true in any C*-algebra, but we only prove here the von Neumann algebra case, as it is all that will be needed later on.  
\begin{proof} For $n=1$ and $a$ selfadjoint, this is \cite[Lemma 3.4]{NRS}. The same proof holds in this case with the obvious modifications.
	
Let $0\leq z\leq 1$ be a central. Using Borel functional calculus, we can write $z$ as a norm limit of elements of the form $\sum_{k=1}^N t_ke_k$, where the $(e_k)_{k=1}^N$ are central pairwise orthogonal projections adding up to 1, and $t_k\in [0,1]$ for all $k$. It thus suffices to assume that $z$ is exactly of this form. 
	
We have $\bm b\in D_A(\bm a)$ if and only  $e_k\bm b\in D_{e_kA}(e_k\bm a)$ for all $k$. This holds since
 unitary mixing operators on $A\cong \bigoplus_{k=1}^N e_kA$  are of the form $\bigoplus_{k=1}^N T_k$, with
$T_k\in \Mix(e_kA)$ for all $k$. Since $e_kz$ is a scalar multiple of $e_k$ (the unit of $e_kA$), cutting down by each $e_k$ reduces the proof to the case that $z$ is a scalar. In this case, the result is simply the convexity of $D_A(\bm a)$.
\end{proof}

In the next section we make essential use of the Zsido's Approximation Lemma  from \cite{zsido}:
\begin{lemma}\label{zsidolemma}
Let $A$ be a simple C*-algebra without bounded traces. Let $\rho\in \St(A)$. Then the operator $A\ni a\mapsto \rho(a)1\in A$ is in the point-norm closure of $\Mix(A)$. 	
\end{lemma}

\section{Proofs of Theorem \ref{zeroaverage} and Corollary \ref{Hmaps}}
Let us introduce some notation: We denote by $\T(A)$ the set of tracial states of $A$. Given an ideal $I\subseteq A$, we denote by  $\St(A)_I$ the states on $A$ that vanish on $I$, i.e., that factor through $A/I$.  Given a state $\rho\in \St(A)$ and $\bm a\in A^n$ we evaluate $\rho$ on $\bm a$ coordinatewise:
\[
\rho(\bm a)=(\rho(a_1),\ldots,\rho(a_n))\in \C^n.\] 
We regard $\C^n$ endowed with the maximum norm. In this way, $A^n\ni \bm a\mapsto \rho(\bm a)\in \C^n$ has norm $1$.

We shall deduce the results stated in the introduction from the following result, of independent interest:

\begin{theorem}\label{0inDas}
Let $\bm a_1,\ldots,\bm a_m\in A^n$. Let $r\geq 0$. The following are equivalent:
\begin{enumerate}[(i)]
\item
$\inf \{\|\sum_{i=1}^m T_i(\bm a_i)\|: T_i\in \Mix(A)\}\leq r$.
\item
The following conditions hold: 
\begin{enumerate}[(a)]
\item
$\|\sum_{i=1}^m \tau(\bm a_i)\|\leq r$ for all $\tau\in \T(A)$,
\item
for each ideal $M\in \Max(A)$ there exist states $\rho_1,\ldots,\rho_m\in \St(A)_M$ such that 
$\|\sum_{i=1}^m \rho_i(\bm a_i)\|\leq r$.
\end{enumerate}
\end{enumerate}
\end{theorem}

\begin{proof}[Proof of (i) $\Rightarrow$ (ii)]
Let $\tau\in \T(A)$. Since $\tau\circ T=\tau$	for all $T\in \Mix(A)$, it follows that
\[
\sum_{i=1}^m \tau(\bm a_i) = \tau(\sum_{i=1}^n T_i(\bm a_i))
\]
for all $T_i\in \Mix(A)$. Since $\tau$ is a state, it is clear that (i) implies 
that the left hand side of the equation above has norm at most $r$. (Recall that 
we've endowed $\C^n$ with the maximum norm.) This proves (a).

Let $M\in \Max(A)$ and $\rho\in \St(A)_M$. Then $\rho \circ T\in \St(A)_M$ for all $T\in \Mix(A)$. 
Also, 
\[
\sum_{i=1}^m (\rho\circ T_i)(\bm a_i) = \rho(\sum_{i=1}^n T_i(\bm a_i)).
\]
Passing to the infimum over all $T_1,\ldots,T_n\in \Mix(A)$ and using (i), we deduce that 
\[
\inf \{\|\sum_{i=1}^m \rho_i(\bm a_i)\|:\rho_i\in \St(A)_M\}\leq r.
\] 
By the compactness of $\St(A)_M$ in the weak* topology, this infimum is attained. This proves
(b).
\end{proof}

Before proving  (ii) $\Rightarrow$ (i) of Theorem \ref{0inDas}, we use a standard Hahn-Banach/Kaplansky density argument to reduce the proof to the  von Neumann algebra case.  
\begin{lemma}\label{bidualreduction}
	Let $\bm a_1,\ldots\bm a_m\in A^n$. Regard $A$ as a C*-subalgebra of its bidual $A^{**}$. Then
	\[
	\inf \{\|\sum_{i=1}^m T_i(\bm a_i)\|: T_i\in \Mix(A)\}=\inf \{\|\sum_{i=1}^m T_i(\bm a_i)\|: T_i\in \Mix(A^{**})\}.
	\]
\end{lemma}
\begin{proof}
	Clearly, the right side is dominated by the left side. Let $r$  denote the number on the left side. Let $\epsilon>0$.  Suppose that $\bm b=\sum_{i=1}^m \bm b_i$ is such that $\|\bm b\|<r+\epsilon$, with
	\[
	\bm b_i=\sum_{k=1}^{n_i} t_{i,k}u_{i,k}\bm a_i u_{i,k}^*,
	\] 
	where the sum is a convex combination, and where $u_{i,k}\in U(A^{**})$  for all $i,k$. By Kaplansky's density theorem for unitaries (\cite[Theorem 2]{GlimmKad}), there exist (commonly indexed) nets of unitaries  $(u_{i,k,\lambda})_{\lambda}\in U(A)$ such that  $u_{i,k,\lambda}\to u_{i,k}$ in the ultrastrong* topology for $i=1,\ldots,m$. For each $i=1,\ldots,m$ and $\lambda$, set
	\[
	\bm b_{i,\lambda} = \sum_{k=1}^{n_i} t_iu_{i,k,\lambda}\bm a_i u_{i,k,\lambda}^*\in D_A(\bm a_i).
	\]
	Then $\bm b_{i,\lambda}\to \bm b_i$ coordinatewise in the weak* topology $\sigma(A^{**},A^*)$, equivalently, in the $\sigma((A^n)^{**},(A^n)^*)$ topology. Consider the set
	\[
	S=\mathrm{co}\{\sum_{k=1}^m\bm b_{i,\lambda}:\lambda\}.
	\]
	This is a convex subset of $A^n$ whose $\sigma((A^n)^{**},(A^n)^*)$ closure in $(A^{**})^n$ contains $\bm b$. It follows that $S$ must intersect the ball $\{\bm x\in A^n:\|\bm x\|<r+\epsilon\}$. For suppose that this is not the case. Then, by the Hahn-Banach theorem, there exists $\bm{\rho}\in (A^n)^*$ such that $\mathrm{Re}(\bm{\rho},\bm{x})<r+\epsilon$ for all $\|\bm{x}\|<r+\epsilon$, 
	while $\mathrm{Re}(\bm{\rho},\bm y)>r+\epsilon$ for all  $\bm{y}\in S$. The first inequality  implies that $\|\bm \rho\|\leq 1$ and the second one that $\mathrm{Re}(\bm{\rho},\bm b)\geq r+\epsilon$. This contradicts that $\|\bm b\|<r+\epsilon$. Thus, there exists a convex combination of sums of the form $\sum_{k=1}^m\bm b_{i,\lambda}$
	with norm $<r+\epsilon$. This yields an element of norm $<r+\epsilon$ and of the desired form. 
\end{proof}

\begin{proof}[Proof of Theorem \ref{0inDas} (ii) $\Rightarrow$ (i)]
	
The proof proceeds in stages: We first obtain the result in the case that $A$ is a simple C*-algebra without  bounded traces. Next,  we deal with the case that  $A$ is a von Neumann algebra, which we break-up into the finite and the properly infinite case. Finally, relying on the previous lemma, we deal with the general case.

\emph{$A$ is simple and traceless}. Let $\bm a_1,\ldots,\bm a_m\in A^n$. Let $\epsilon>0$. By assumption, there exist 	$\rho_1,\ldots,\rho_m\in \St(A)$
such that 
\[
\|\sum_{i=1}^m \rho_i(\bm a_i) \|\leq r.
\]
By Lemma \ref{zsidolemma}, for each $i$ the map $A\ni a\mapsto \rho_i(a)1\in A$ is a point-norm limit of unitary mixing operators. Then, letting $T_{i,\lambda}\in \Mix(A)$ be such that $\lim_\lambda T_{i,\lambda} = \rho_i 1$,
we deduce at once that 
\[
\|\sum_{i=1}^m T_{i,\lambda}(\bm a_i)\|< r+\epsilon,
\]
for some $\lambda$. This, (ii) implies (i) in this case.

\emph{$A$ is a properly infinite von Neumann algebra}. Let $\bm a_1,\ldots,\bm a_m\in A^n$ be $n$-tuples satisfying condition (b) of Theorem \ref{0inDas} for a given $r\geq 0$.  Let $\epsilon>0$. Let $M$ be a maximal ideal of $A$. Passing to the quotient $A/M$, which is simple and traceless, we know, as established in the previous paragraph, that 
\[
\Big(\sum_{i=1}^m D_{A/M}(\pi_M(\bm a_i))\Big)\cap Z(A/M)^n
\] 
contains an element of norm less than $r+\epsilon$ (where $\pi_M\colon A\to A/M$ denotes the quotient map). Since $D_A(\bm a_i)\cap Z(A)^n$ maps densely onto $D_{A/M}(\pi_M(\bm a_i))\cap Z(A/M)$  (Lemma \ref{centermapsdensely}), there exist  $\bm z_i\in D_A(\bm a_i)\cap Z(A)^n$ such that 
\[
\|\sum_{i=1}^m \pi_M(\bm{z}_i)\|<r+\epsilon.
\] 
Recall that von Neumann algebras are weakly central, i.e., $\Max(A)$ is homeomorphic to the spectrum of $Z(A)$ 
via the map $\Max(A)\ni M\mapsto M\cap Z(A)\in \widehat{Z(A)}$. Identifying in this way $\Max(A)$ with 
$\widehat{Z(A)}$, we can rephrase the inequality above as 
\[
\|\sum_{i=1}^m\widehat{\bm{z}}_i(M)\|<r+\epsilon,
\] 
where $\widehat{z}\colon \widehat{Z(A)}\to \C$ denotes the Gelfand transform of $z\in Z(A)$. This inequality is valid in neighborhood $\mathcal U_M$ of $M$:
\[
\|\sum_{i=1}^m\widehat{\bm{z}}_i(M')\|<r+\epsilon,
\] 
for all $M'\in \mathcal U_M$. Since the spectrum of $Z(A)$ is totally disconnected, we can refine the cover $(\mathcal U_M)_M$ to a  finite partition of $\widehat{Z(A)}$ by clopen sets. We  thus obtain central projections $(e_k)_{k=1}^N$, corresponding to these clopen sets, such that $\sum_{k=1}^N e_k =1$, and for each $k=1,\ldots,N$ there exists $\bm z_{i,k}\in  D_A(\bm a_i)\cap Z(A)^n$ such that 
\begin{equation}\label{gluecentrals}
\|e_k\sum_{i=1}^m\bm z_{i,k}\|<r+\epsilon.
\end{equation}
Now let
\[
\bm z_i=e_1\bm z_{i,1}+\ldots+e_N \bm z_{i,N}\hbox{ for }i=1,\ldots,m.
\]
By  Lemma \ref{centrallyconvex}, $\bm z_i\in D_A(\bm a_i)\cap Z(A)^n$. Further, 
\[
\|\sum_{i=1}^m \bm z_i\|=\max_{1\leq k\leq N} \|e_k\sum_{i=1}^m\bm z_{i,k}\|<r+\epsilon.
\] 
This proves that (ii) implies (i) in the case that $A$ is a properly infinite von Neumann algebra.

\emph{$A$ is a finite von Neumann algebra}. Suppose that $A$ is a finite von Neumann algebra. This case follows using standard results on the center-valued trace.
Let $E\colon A\to Z(A)$ denote the center-valued trace. Let $\bm a_1,\ldots,\bm a_m\in A^n$. Suppose that $\bm a_1,\ldots,\bm a_m$ satisfy condition (a) of Theorem \ref{0inDas} for a given $r\geq 0$. For each point evaluation
$\mathrm{ev}_M\colon Z(A)\to \C$, with $M\in \widehat{Z(A)}$, the map $\mathrm{ev}_M\circ E$ is a trace on $A$. We deduce at once from condition (a)
that
\[
\|\sum_{i=1}^m E(\bm a_i)\|\leq r.
\]
It thus suffices to show that $E$ is a point-norm limit of unitary mixing operators.
Indeed, by \cite[Lemma 8.3.4]{kad-ring}, there exists a net $(T_\lambda)_\lambda\in \Mix(A)$ such that $T_\lambda a\in D_A(a)\cap Z(A)$ for all $a\in A$. Moreover, since $A$ is a finite von Neumann algebra, it has the singleton Dixmier property, i.e., $D_A(a)\cap Z(A)=\{E(a)\}$ for all $a\in A$.

\emph{$A$ is an arbitrary von Neumann algebra}. Suppose now that $A$ is an arbitrary von Neumann algebra. 
Then there exists a central projection $e$ such that $eA$ is a finite von Neumann algebra, while  $(1-e)A$ is properly infinite. Let $\bm{a}_1,\ldots,\bm a_m\in A^n$ be $n$-tuples satisfying (a) and (b) of Theorem \ref{0inDas}. Then $e\bm{a_1},\ldots,e\bm a_m$ satisfy (a) relative to $eA$, and as demostrated above this implies that $\|\sum_{i=1}^m T_i\bm a_i\|<r+\epsilon$ for some 
$T_1,\ldots T_m\in \Mix(eA)$. On the other hand, $(1-e)\bm{a_1},\ldots,(1-e)\bm{a_m}$ satisfy condition (b) on $(1-e)A$, and as demonstrated above, $\|\sum_{i=1}^m T_i'(\bm a_i)\|<r+\epsilon$ form some  $T_i'\in \Mix((1-e)A)$. Since the operators  $a\mapsto T_i(ea)+T_i'((1-e)a)$ are unitary mixing operator on $A$ for all $i$, we again deduce that (ii) implies (i) for an arbitrary von Neumann algebra.  

\emph{Case of an arbitrary C*-algebra}. 
Let us argue that if  $A$ is unitally embedded in some C*-algebra $B$, then conditions  (a) and (b) are verified relative to $B$ as well. Since every bounded trace on $B$ restricts to a trace on $A$,  we immediately deduce that (a)
holds in $B$.  Let $M\subseteq B$ be a maximal ideal of $B$. By condition (b) applied in $A$, there exist states 
$\rho_1,\ldots,\rho_m\in \St(A)$ that vanish on $A\cap M$ and such that 
\[
\|\sum_{i=1}^m \rho_i(\bm a_i)\|\leq r.
\]
By the Hahn-Banach theorem, the  state induced by $\rho_i$ on $A/(A\cap M)$ extends to a state on $B/M$. We thus obtain
states  $\tilde \rho_1,\ldots,\tilde\rho_m\in \St(B)_M$ extending $\rho_1,\ldots,\rho_m$. This yields condition (b)
relative to $B$. Let us now regard $A$ as a C*-subalgebra of $A^{**}$. Since the latter is a von Neumann algebra,
we already know that (ii) implies (i) in this case. Applying Lemma \ref{bidualreduction}, we deduce (i). 
\end{proof}

Let us now prove the results stated in the introduction:

\begin{proof}[Proof of Theorem \ref{zeroaverage}]
Let us prove that (i) implies (iii). Let $F=\{a_1,\ldots,a_n\}$ be a finite subset of $V$, and let $\epsilon>0$.
Set $\bm a=(a_1,\ldots,a_n)$. Then $\bm a$ satisfies conditions (a) and (b) of Theorem \ref{0inDas} (ii), with $m=1$ and $r=0$. We conclude that  $\bm 0\in D_A(\bm a)$. Thus, there exists $T_{F,\epsilon}\in \Mix(A)$ such that $\|T_{F,\epsilon}a_i\|<\epsilon$ for all $i$. The net of operators $(F,\epsilon)\mapsto T_{F,\epsilon}$ is then as desired.

It is clear that (iii) implies (ii). 

Let us prove that (ii) implies (i). Let $V$ be as in (ii). Observe first that  $a-Ta\in [A,A]$ for any $T\in \Mix(A)$. Hence,  if $T_na\to 0$ for some sequence $T_n\in \Mix(A)$, then $a\in \overline{[A,A]}$. This shows that  $V\subseteq \overline{[A,A]}$. 

Let $M$ be a maximal ideal. Let $F=\{a_1,\ldots,a_n\}$ be a finite subset of $V$. Suppose, for the sake of contradiction, that  $\bm 0\notin \{\rho(\bm a):\rho\in \St(A)_M\}$. Then, by the Hahn-Banach Theorem applied to the compact convex set $\{\rho(\bm a):\rho\in \St(A)_M\}$  and $\bm 0\in \C^n$, there exist 
$\alpha_1,\ldots,\alpha_n\in \C^n$  and $c>0$ such that 
\[
\mathrm{Re}(\sum_{i=1}^n \alpha_i \beta_i)\geq c\hbox{ for all }(\beta_1,\ldots,\beta_n)\in \{\rho(\bm a):\rho\in \St(A)_M\}.
\]
Let $a=\sum_{i=1}^n \alpha_i a_i$.
Then, $0\notin \{\rho(a):\rho\in \St(A)_M\}$, which in turn implies that $0\notin D_A(a)$, by the equivalence
of (i) and (iii) for the case $n=1$ (\cite[Theorem 4.7]{NRS}). This contradicts (ii). We have shown that for every finite set
$F\subseteq V$ the set $\{\rho\in \St(A)_M:\rho(F)=\{0\}\}$ is non-empty. By the compactness of $\St(A)_M$ in the weak* topology, we obtain $\rho\in \St(A)_M$ such that $\rho(V)=\{0\}$, as desired.
\end{proof}	

\begin{remark}
If every quotient of  $A$ has a bounded trace, then Theorem \ref{zeroaverage} (i) simply asserts that $V\subseteq \overline{[A,A]}$. In this case, it is not difficult to go from (i)$\Leftrightarrow$(ii) to (iii) by a process of successive averagings towards zero of a given collection of elements: starting with $a_1,a_2,\ldots, a_n$ in $\overline{[A,A]}$, we argue by the equivalence of (i) and (ii) that there exists $T_1\in \Mix(A)$ such that $\|T_1a_1\|<1/2$. Since $T_1a_2\in \overline{[A,A]}$,   we can choose $T_2\in \Mix(A)$ such that $\|T_2T_1a\|<\frac1{2^2}$, etc.  This simple strategy breaks down when the C*-algebra $A$ has traceless quotients, since after the first step there is no guarantee that $0\in D_A(T_1a_2)$.
\end{remark}	

\begin{proof}[Proof of Corollary \ref{Hmaps}]
Let $V=\{a-H(a):a\in A\}$. By the assumption that $H$ preserves traces, $\tau|_V=0$ for any bounded trace $\tau$.
Since $\overline{[A,A]}=\bigcap_{\tau\in \T(A)} \ker \tau$ (see \cite{CuPed}), it follows that $V\subseteq \overline{[A,A]}$.

Let $M$ be a maximal ideal and denote by $\pi_M\colon A\to A/M$ the quotient map. Since $A/M$ is simple and unital, its center is isomorphic to $\C$. We may thus regard the restriction of $\pi_M$ to $Z(A)$ as a homomorphism onto $\C$. Set  $\rho_M=\pi_M\circ H$. Then $\rho_M$ is a state factoring through $A/M$ and
such that $\rho_M(a-H(a))=0$ for all $a\in A$, i.e., $V\subseteq \ker \rho_M$. It follows by Theorem \ref{zeroaverage}
that there exists a net of unitary mixing operators $(T_\lambda)_\lambda$ such that $T_{\lambda}(a-H(a))\to 0$
for all $a\in A$. Since the $T_\lambda$s fix the center, we obtain that $T_\lambda a\to Ha$ for all $a\in A$, as desired.
\end{proof}

\end{document}